\documentclass[aos,preprint]{imsart}
\RequirePackage[OT1]{fontenc}
\RequirePackage{amsthm,amsmath,amsfonts,amssymb}
\RequirePackage[numbers]{natbib}
\RequirePackage[colorlinks,citecolor=blue,urlcolor=blue]{hyperref}
\usepackage{enumerate}

\startlocaldefs
\numberwithin{equation}{section}
\theoremstyle{plain}
\newtheorem{thm}{Theorem}[section]
\newtheorem{lem}[thm]{Lemma}
\newtheorem{prop}[thm]{Proposition}
\newtheorem{cor}[thm]{Corollary}
\newtheorem{ex}[thm]{Example}
\DeclareMathOperator{\oddsr}{or}

\newcommand{\E}{\mathbb{E}}

\renewcommand{\P}{\mathbb{P}}

\newcommand{\R}{\mathbb{R}}

\newcommand{\cX}{\mathcal{X}}
\newcommand{\cZ}{\mathcal{Z}}

\newcommand{\indic}{{1\!\!1}}

\newcommand{\mtp}{{\rm MTP}_{2}}
\newcommand{\tp}{{\rm TP}_{2}}
\newcommand{\cov}{{\rm cov}}
\newcommand{\cd}{\,|\,}
\newcommand{\cip}{\mbox{\,$\perp\!\!\!\perp$\,}}
\newcommand{\indep}{\cip}
\newcommand{\gse}{\,\mbox{$\perp_{G}$}\,}
\newcommand{\dd}{\mathrm{d}}
\def\proofname{\indent {\sl Proof}}

\newcommand\notindependent{\!\perp\!\!\!\!\not\perp\!}

\endlocaldefs

\begin{document}

\begin{frontmatter}
\title{Total positivity in Markov structures\thanksref{T1}}
\runtitle{Total positivity in Markov structures}
\thankstext{T1}{SF was supported in part by an NSERC Discovery Research Grant, KS by grant \#FA9550-12-1-0392 from the U.S. Air Force Office of Scientific Research (AFOSR) and the Defense
Advanced Research Projects Agency (DARPA), CU by the Austrian Science Fund (FWF) Y 903-N35, and PZ by the European Union Seventh Framework Programme PIOF-GA-2011-300975.}

\begin{aug}
\author{\fnms{Shaun} \snm{Fallat}\thanksref{m1}\ead[label=e1]{shaun.fallat@uregina.ca}},
\author{\fnms{Steffen} \snm{Lauritzen}\thanksref{m2}\ead[label=e2]{lauritzen@math.ku.dk}}
\author{\fnms{Kayvan} \snm{Sadeghi}\thanksref{m3}
\ead[label=e3]{k.sadeghi@statslab.cam.ac.uk}}
%\ead[label=u1,url]{http://www.foo.com}}
\author{\fnms{Caroline} \snm{Uhler}\thanksref{m4}
\ead[label=e4]{cuhler@mit.edu}}
\author{\fnms{Nanny} \snm{Wermuth}\thanksref{m5}
\ead[label=e5]{wermuth@chalmers.se}}
\and
\author{\fnms{Piotr} \snm{Zwiernik}\thanksref{m6}
\ead[label=e6]{piotr.zwiernik@upf.edu}}

\runauthor{S. Fallat et al.}

\affiliation{University of Regina\thanksmark{m1}, University of Copenhagen\thanksmark{m2}, University of Cambridge\thanksmark{m3}, Massachusetts Institute of Technology\thanksmark{m4}, IST Austria\thanksmark{m4}, Chalmers University of Technology\thanksmark{m5}, Johannes Gutenberg-University\thanksmark{m5}, and Universitat Pompeu Fabra\thanksmark{m6}}

\address{Department of Mathematics and Statistics\\
University of Regina\\
Regina, SK, Canada\\
\printead{e1}}

\address{Department of Mathematical Sciences\\
University of Copenhagen\\
Copenhagen, Denmark\\
\printead{e2}}

\address{Statistical Laboratory\\
University of Cambridge\\
Cambridge, UK\\
\printead{e3}}

\address{Massachusetts Institute of Technology\\
 Cambridge, USA\\
and Institute of Science and Technology Austria\\
Klosterneuburg, Austria\\
\printead{e4}}

\address{Chalmers University of Technology\\
Gothenburg, Sweden\\
and Johannes Gutenberg-University\\
Mainz, Germany\\
\printead{e5}}

\address{Universitat Pompeu Fabra\\
and Barcelona GSE\\
Barcelona, Spain\\
\printead{e6}}
\end{aug}

\begin{abstract}
We discuss properties of distributions that are multivariate totally positive of order two ($\mtp$) related to conditional independence. In particular, we show that any independence model generated by an $\mtp$ distribution is a compositional semigraphoid which is upward-stable and singleton-transitive. In addition, we prove that any $\mtp$ distribution satisfying an appropriate support condition is faithful to its concentration graph. Finally, we analyze factorization properties of $\mtp$ distributions and discuss ways of constructing $\mtp$ distributions; in particular we give conditions on the log-linear parameters of a discrete distribution which ensure $\mtp$ and characterize conditional Gaussian distributions which satisfy $\mtp$.
\end{abstract}

\begin{keyword}[class=MSC]
\kwd[Primary ]{60E15}
\kwd{62H99}
\kwd[; secondary ]{15B48}
\end{keyword}

\begin{keyword}
\kwd{Association}
\kwd{concentration graph}
\kwd{conditional Gaussian distribution}
\kwd{faithfulness}
\kwd{graphical models}
\kwd{log-linear interactions}
\kwd{Markov property}
\kwd{positive dependence}
\end{keyword}

\end{frontmatter}

\section{Introduction}
\label{sec:introduction}
%\subsection{}

This paper discusses a special form of positive dependence. \emph{Positive dependence} may refer to two random variables that have a positive covariance,  but other definitions of positive dependence have been proposed as well; see \cite{lehmann1966} for an overview. Random variables $X=(X_1,\ldots,X_d)$ are said to be \emph{associated} if ${\rm cov}\{f(X),g(X)\}\geq 0$ for any two non-decreasing functions $f$ and $g$ for which  $\E|f(X)|$,  $\E|g(X)|$, and $\E|f(X)g(X)|$ all exist \citep{esary1967association}. This notion has important applications in probability theory and statistical physics; see, for example,~\cite{newman1983general,Newman1984}.

However, association may be difficult to verify in a specific context. The celebrated FKG theorem, formulated by Fortuin, Kasteleyn, and Ginibre in \cite{fortuin1971correlation}, introduces an alternative notion and establishes that $X$ are  associated if their joint density function is \textit{multivariate totally positive of order~2}: A function $f$ over $\cX=\prod_{v\in V}\cX_{v}$, where each $\cX_v$ is totally ordered, is \emph{multivariate totally positive of order two} ($\mtp$) if
$$
f(x)f(y)\quad\leq \quad f(x\wedge y)f(x\vee y)\qquad\mbox{for all }x,y\in \cX,
$$
where $x\wedge y$ and $x\vee y$ denote the element-wise minimum and maximum.

These inequalities are often easier to check. Furthermore, most other known definitions of positive dependence are implied by the $\mtp$ constraints; see for example~\cite{colangelo2005some} for a recent overview. Note that the conditions are on the probabilities or the density and not on other types of traditional measures of dependence. But as we shall see, the above inequality constraints combined with conditional independence restrictions specify positive associations along edges in undirected graphs, named and studied as dependence graphs or concentration graphs; see for instance~\cite{lau96, W15}.

$\mtp$ distributions have also played an important role in the study of ferromagnetic Ising models, i.e.~distributions of binary variables where all interaction potentials are pairwise and non-negative. It has been noted in \cite{propp1996exact} that the block Gibbs sampler is monotonic if the target distribution is $\mtp$, and hence particularly efficient in this setting. Bartolucci and Besag~\cite[Section 5]{bartolucci2002recursive} showed that much of this work can in fact be extended to arbitrary binary Markov fields. See also \cite{djolonga2015scalable} for an optimization viewpoint. The special case of Gaussian distributions was studied by Karlin and Rinott~\cite{karlinGaussian} and very recently by Slawski and Hein~\cite{slawski2015estimation} from a machine learning perspective.

Consequences of $\mtp$ distributions with respect to marginal and mutual independences
were studied by Lebowitz \cite{lebowitz1972bounds} and Newman \cite{Newman1984}. They showed in particular that independence of two components of a random vector with an $\mtp$ distribution, is equivalent to a block-diagonal structure in the covariance matrix and
that mutual independence of several components can be inferred from a block-diagonal
covariance matrix (see also Theorem~\ref{th:newman} and Theorem~\ref{th:block} below). This is remarkable because covariances and correlations are the weakest types of measures of dependence, see \cite{XMG}; although they identify independence in Gaussian distributions,
this is often not the case for other types of distribution. %but see here Theorem~\ref{th:Faith} below.

In this paper, we discuss implications of the $\mtp$ constraints for conditional independence and vice versa. There is some related work in the context of copulae \cite{muller2005archimedean}. Our paper can be seen as a continuation of work by Sarkar~\cite{sarkar} and, in particular, by Karlin and Rinott~\cite{KarlinRinott80, karlinGaussian}. They noted that the family of $\mtp$ distributions is stable with respect
to forming marginal and conditional distributions. At least as important is that they give
constraints on different types of measures of dependence needed to verify the $\mtp$ property of a joint distribution for discrete and for Gaussian random variables.%~\cite[p.~469]{KarlinRinott80}

The $\mtp$ property may appear extremely restrictive when higher order interactions
are needed to capture the relevant types of conditional dependence or when distributions
are studied which do not satisfy any conditional independence constraints. However, as
we shall see, the $\mtp$ constraints become less restrictive when imposing an additional Markov structure. For example, all finite dimensional distributions of a Markov chain are $\mtp$ whenever all $2\times 2$ minors of its transition matrix are non-negative~\cite[Proposition 3.10]{KarlinRinott80}. This result holds true also for non-homogeneous Markov chains. Moreover, models with latent, that is hidden or unobserved, variables may be $\mtp$. For example, factor analysis models with a single factor are $\mtp$ when each observed variable has an, albeit unobserved, positive dependence on the single, hidden factor~\cite{wermuth2014star}. Similar statements apply to binary latent class models~\cite{ARSZ, holland1986, wermuth2014star} and to latent tree models, both in the Gaussian and in the binary setting~\cite{Steel20091141, LTbook}. Furthermore, many data sets can be well explained or modelled assuming that the generating distribution is $\mtp$ or nearly $\mtp$; see Section~\ref{sec:examples} for some examples and also the discussion in Section~\ref{sec:discussion}.

The paper is organized as follows: In Section~\ref{sec:notation} we introduce our notation and provide the main definitions. In Section~\ref{sec:MTP} we review basic properties of $\mtp$ distributions and discuss the link to positive dependence and independence structures. In Section~\ref{sec:examples} we concentrate on the $\mtp$ condition in the Gaussian and binary setting and provide several data examples where the $\mtp$ property appears naturally. Section~\ref{sec:CIMTP} analyzes $\mtp$ distributions with respect to conditional independence relations. One of the main results in this paper is Theorem~\ref{prop:intersection}, which shows that any independence model generated by an $\mtp$ distribution is a  singleton-transitive compositional semigraphoid which is also upward-stable;  the latter means that new arbitrary elements can be added to the conditioning set of every existing independence statement without violating independence. Theorem~\ref{th:block} gives a complete characterization of the marginal independence structures of $\mtp$ distributions. In Section~\ref{sec:Faithful}, we study Markov properties of $\mtp$ distributions and show that such distributions are always faithful to their concentration graph. In Section~\ref{sec:factorization}, we analyze factorization properties of $\mtp$ distributions, show how to use these properties to build $\mtp$ distributions from smaller $\mtp$ distributions,  briefly discuss log-linear expansions of discrete $\mtp$ distributions, and give conditions for conditional Gaussian distributions to satisfy the $\mtp$ constraints. We conclude our paper with a short discussion in Section~\ref{sec:discussion}.

\section{Preliminaries and notation}\label{sec:notation}
Let $V$ be a finite set and let $X=(X_v,v\in V)$ be a  random vector. We consider the product space $\cX=\prod_{v\in V}\cX_{v}$, where $\cX_v\subseteq \R$ is the state space of $X_v$, inheriting the order from $\R$.
In this paper, the state spaces are either discrete (finite sets) or open intervals on the real line. Hence, we can partition the set of variables as $V=\Delta\cup\Gamma$, where $\cX_{v}$ is discrete if $v\in \Delta$, and $\cX_{v}$ is an open interval if  $v\in \Gamma$. For $A\subseteq V$ we further write $X_A=(X_v)_{v\in A}$, $\cX_A=\times_{v\in A}\cX_v$ and so on.

All distributions are assumed to have densities with respect to the product measure $\mu=\otimes_{v\in V}\mu_v$,  where $\mu_v$ is the counting measure for $v\in\Delta$, and $\mu_v$ is the Lebesgue measure giving length 1 to the unit interval for $v\in \Gamma$. We shall refer to $\mu$ as the \emph{standard base measure}. Similarly to the above, we write $\mu_A=\otimes_{v\in A}\mu_v$.

Finally, we introduce some definitions related to graphs: An \emph{undirected graph} $G = (V, E)$ (sometimes also called \emph{concentration graph} in the literature of graphical models) consists of a nonempty set of \emph{vertices} or \emph{nodes} $V$ and a set of undirected edges $E$. Our graphs are \emph{simple} meaning that they have no self-loops and no multiple edges. We write $uv$ for an edge between $u$ and $v$ and say that the vertices $u$ and $v$ are \emph{adjacent}.  A \emph{path} in $G$ is a sequence of nodes $(v_0,v_1,\dots, v_k)$ such that $v_iv_{i+1}\in E$ for all $i=0,\dots, k-1$ and no node is repeated, i.e., $v_i\neq v_j$ for all $i,j\in\{0,1,\dots , k\}$ with $i\neq j$. Thus an edge is the shortest type of path. A \emph{cycle} is a path with the modification that $v_0=v_k$. Furthermore, we say that two distinct nodes $u, v\in V$ are \emph{connected} if there is a path between $u$ and $v$; a graph is \emph{connected} if all pairs of distinct nodes are connected. A graph is \emph{complete} if all possible edges are present. In addition, two subsets $A,B\subseteq V$ are \emph{separated} by $S\subset V\setminus(A\cup B)$ if every path between $A$ and $B$ passes through a node in $S$. A subgraph of $G$, \emph{induced} by a set $A\subset V$, consists of the nodes in $A$ and of the edges in $G$ between nodes in $A$. Finally, a maximal complete subgraph is a \emph{clique}.

\section{Basic properties and positive dependence}\label{sec:MTP}

We start this section by formally introducing $\mtp$ distributions and  discuss basic properties of these. We define the coordinate-wise minimum and maximum as
$$x\wedge y=(\min(x_{v},y_{v}),v\in V),\quad x\vee y=(\max(x_{v},y_{v}),v\in V).$$
 A function $f$ on $\cX$ is said to be \emph{multivariate totally positive of order two} ($\mtp$) if
\begin{equation}\label{eq:MTP2}
f(x)f(y)\quad\leq \quad f(x\wedge y)f(x\vee y)\qquad\mbox{for all }x,y\in \cX.
\end{equation}
For $|V|=2$, a function that is $\mtp$ is simply called \emph{totally positive of order two} ($\tp$)~\cite{KarlinRinott80}. Let $X=(X_v,v\in V)$ have density function $f$ with respect to the standard base measure $\mu$. Then we say that $X$ or the distribution of $X$ is $\mtp$ if its density function $f$ is $\mtp$. Note that this concept is well-defined since $\cX_v$ is either discrete or an open interval on the real line.

A basic property of $\mtp$ distributions is that it is preserved under increasing coordinatewise transformations. We begin with a simple result for strictly increasing functions.

\begin{prop}\label{prop:differentiable}
Let $X$ be a random vector taking values in $\cX$. Let $\phi=(\phi_{v},v\in V)$ be such that  $\phi_{v}:\R  \to \R$ are strictly increasing and differentiable for all $v\in V$. If the distribution of $X$ on $\mathcal{X}$ is $\mtp$, then the distribution of $Y=\phi(X)$  is $\mtp$.
\end{prop}
\begin{proof}
We use the following fact from~\cite[Equation (1.13)]{KarlinRinott80}: Let $a_v:\R\to \mathbb{R}$ be positive and let $b_v:\R\to \mathbb{R}$ be non-decreasing. If $f:\R^{V}\to \R$ is $\mtp$, then the function
\begin{equation}\label{eq:prodtrans}
g(y)=f\{b_{v}(y_{v}),v\in V\}\prod_{v\in V}a_{v}(y_{v}),\end{equation}
is $\mtp$. Let $b_v(y_v)=\phi_v^{-1}(y_v)$ and let $a_v(y_v) = 1/\phi'_v(\phi_v^{-1}(y_v))$, where $\phi'_v(y_v)$ denotes the first derivative of $\phi_v$. Then $g(y)$ is the density of $Y=\phi(X)$ and we obtain from (\ref{eq:prodtrans}) that $Y$ is $\mtp$.
\end{proof}

We say that a function $\phi_v(x):\cX_v\to\R$  is \emph{piecewise constant} if $\phi_v(\cX_v)$ is finite and we can then similarly show that the
$\mtp$ property is preserved under transformations which are piecewise constant and non-decreasing.

\begin{prop}\label{prop:monotone}
Let $X$ be a random vector taking values in  $\cX$ as before. For $A\subseteq V$, let $\phi=(\phi_{v},v\in V)$ be such that  $\phi_{v}:\cX_v \to \R$ is piecewise constant and non-decreasing for all $v\in A$ and $\phi_v(x_v)=x_v$ for $v\not \in A$. If the distribution of $X$ is $\mtp$, then the distribution of $Y=\phi(X)$  is $\mtp$.
\end{prop}

\begin{proof}
If $f$ denotes the density function for $X$ and $g$ the density function for $Y$, both w.r.t.\ a standard base measure, we have that
\[g(y)=\int_{\phi_A^{-1}(y_A)}f(x)\,\dd\mu_A(x_A).\]
 Since $\phi$ is non-decreasing, we have
$$\phi_A^{-1}(y^1_A) \wedge \phi_A^{-1}(y^2_A) = \phi_A^{-1}(y^1_A\wedge y_A^2), \quad \phi_A^{-1}(y_A^1) \vee \phi_A^{-1}(y_A^2) = \phi_A^{-1}(y_A^1\vee y_A^2),$$
where for two sets $A, B$,
$$A \wedge B = \{a\wedge b \mid a\in A, b\in B\}, \quad A \vee B = \{a\vee b \mid a\in A, b\in B\}.$$
Hence, we can apply~\cite[Corollary 2.1]{KarlinRinott80} to obtain that $Y$ is $\mtp$.
\end{proof}

\begin{cor}\label{cor:monotone}
Let $X$ be a random vector taking values in  $\cX$ as before and $\phi=(\phi_{v},v\in V)$ be such that  $\phi_{v}:\cX_v \to \R$ is piecewise constant and non-decreasing for all $v\in A$ and $\phi_v(x_v)$ is strictly increasing and differentiable for $v\not \in A$. If the distribution of $X$ is $\mtp$, then the distribution of $Y=\phi(X)$  is $\mtp$.
\end{cor}
\begin{proof}Just combine Proposition~\ref{prop:differentiable} and Proposition~\ref{prop:monotone} by first transforming to $Z_v=X_v$ for $v\in A$ and $Z_v=\phi_v(X_v)$ for $v\not \in A$; then subsequently letting $Y_v=\phi_v(Z_v)$ for $v\in A$ and $Y_v=Z_v$ for $v\not \in A$.
\end{proof}

The following result establishes that the $\mtp$ property is preserved under conditioning, marginalization, and monotone coarsening. A \emph{monotone coarsening} is an operation on a finite discrete state space $\cX_{i}$ that identifies a collection of neighbouring (in the given total order) states. For example, if $\cX_{i} = \{i_1,\dots ,i_p\}$ then $\cX'_{i} = \{\{i_1,\dots ,i_j\},i_{j+1}, \dots, i_k,\{i_{k+1},\dots, i_p\}\}$ is a monotone coarsening. %Furthermore, we say that a conditional distribution is $\mtp$ if it is  $\mtp$ for every value of the conditioning variables.

\begin{prop}\label{prop:ClosedMC}
The $\mtp$ property is closed under conditioning, marginalization, and monotone coarsening. More precisely,
\begin{itemize}
%\item[(i)] If $f$ and $g$ are  $\mtp$ on $\cX$, then their product $f\cdot g$ is $\mtp$;
\item[(i)] If $X$ has an $\mtp$ distribution, then for every $C\subseteq V$, the conditional distribution of $X_C\cd X_{V\setminus C}=x_{V\setminus C}$ is $\mtp$ for almost all $x_{V\setminus C}$;
\item[(ii)] If $X$ has an $\mtp$ distribution, then for every $A\subseteq V$, the marginal distribution $X_A$ of $X$ is $\mtp$;
\item[(iii)] If $X$ is $\mtp$ and discrete, and $Y$ is obtained from $X$ by monotone coarsening, then $Y$ is $\mtp$.
\end{itemize}
\end{prop}

\begin{proof}

Property (i) follows directly from the definition of $\mtp$. Property (ii) is shown in~\cite[Proposition 3.2]{KarlinRinott80}. Property (iii) is an instance of a non-decreasing and piecewise constant transformation and follows from Proposition~\ref{prop:monotone}.
\end{proof}

As we will see in the following, %Section~\ref{sec:CIMTP}, Section~\ref{sec:Faithful}, and Section~\ref{sec:factorization},
the properties (i) and (ii) are the fundamental building blocks for understanding the implications of $\mtp$ on Markov properties and vice versa. Property (iii) has direct relevance for applications. In the statistical literature it is often warned that dependence relations may get distorted when combining neighboring levels of discrete variables, see for instance~\cite{RAS}. This may still be true for $\mtp$ distributions, see Example~\ref{ex:distortion} below, but the coarsening property (iii) implies that associations cannot become negative by such a process.

Another interesting fact about the $\mtp$ property is that, under suitable support conditions, it is a pairwise property meaning that it can be checked on the level of two variables only, when the remaining variables are fixed. We say that $f$ has \emph{interval support} if for any $x,y\in \cX$ the following holds
\begin{equation}\label{eq:supportCon}
f(x)f(y)\neq 0\qquad\mbox{implies}\qquad f(z)\neq 0\mbox{ for any } \,x\wedge y\leq z\leq x\vee y.
\end{equation}
Note that having \emph{interval support} is equivalent to having full support over a restricted state space that is a product of intervals. In this setting, Karlin and Rinott~\cite[Proposition~2.1]{KarlinRinott80} prove the following result.%This statement can be strengthened further as follows:

\begin{prop}\label{KarlinRinott80}
If $f$ has interval support and  $f:\cX\to \R$ is $\tp$ in every pair of arguments when the remaining arguments are held constant, then $f$ is $\mtp$.
\end{prop}

We conjecture that this result holds also under a weaker support condition, namely that the support is \emph{coordinate-wise connected}~\cite{peters:15}, meaning that the connected components of the support can be joined by axis-parallel lines. We now provide such an instance in the binary $2\times 2\times 2$ setting.

\begin{ex}\rm
\label{ex_coordinatewise}
Consider a binary $2\times 2\times 2$ distribution, where the support only misses the entries $(1,0,0)$ and $(1,0,1)$. In this example there are only two non-trivial pairwise $\tp$ constraints to consider, namely for $x_1=0$ and for $x_2=1$, i.e.\
\begin{equation}
p_{000}p_{011}\geq p_{010}p_{001} \qquad \textrm{and}\qquad p_{010}p_{111}\geq p_{110}p_{011}\label{eq_2inequalities}
\end{equation}
as the remaining six pairwise inequalities reduce to $0\geq 0$.

So assume $p$ satisfies (\ref{eq_2inequalities}).  We need to show that the nine inequalities in (\ref{eq:3binary}) below are satisfied.   Again, six of them are trivial, and two of them are exactly those in (\ref{eq_2inequalities}). The remaining inequality
$$
p_{000}p_{111}\geq p_{001}p_{110},
$$
follows from
multiplying the two inequalities in (\ref{eq_2inequalities}) to get
$$
(p_{000}p_{011})(p_{010}p_{111})\geq (p_{010}p_{001})(p_{110}p_{011}),
$$
and dividing both sides by $p_{010}p_{011}$.
Hence, in this case pairwise $\tp$ constraints imply the $\mtp$ property even though the distribution does not have interval support.
\qed
\end{ex}

As mentioned in Section~\ref{sec:introduction}, if $X$ is $\mtp$, then the variables in $X$ are \emph{associated}, i.e.,
\begin{equation}
\label{eq_increasing}
{\rm cov}\{f(X),g(X)\}\geq 0
\end{equation}
for any coordinatewise non-decreasing functions $f$ and $g$ for which the covariance exists. For discrete distributions this follows by the FKG theorem~\cite{fortuin1971correlation}, or, more generally, by the four functions theorem by Ahlswede and Daykin~\cite{ahlswede1978inequality}; the general case was proved by Sarkar \cite{sarkar}. The following result, first proven by Lebowitz~\cite{lebowitz1972bounds}, shows that the independence structure for associated vectors is encoded in the covariance matrix; see also \cite{joag-dev1983,Newman1984}.

\begin{thm}[Corollary 3, \cite{Newman1984}]\label{th:newman}
If  $X$ are associated and $\E|X_v|^2<\infty$ for all $v\in V$, then $X_A$ is independent of $X_B$ if and only if ${\rm cov}(X_u,X_v)=0$ for all $u\in A$ and $v\in B$.
\end{thm}

In Section~\ref{sec:CIMTP}, we study conditional independence models for $\mtp$ distributions. Interestingly, we will show in Theorem~\ref{th:block} that for $\mtp$ distributions a stronger result holds, namely that every $\mtp$ random vector can be decomposed into independent components such that within each component all variables are mutually marginally dependent. This means in particular that for $\mtp$ distributions, all marginal independence relations also hold when conditioning on further variables; the general version of this property will be termed upward-stability, see Section~\ref{sec:CIMTP}.

\section{Examples of Gaussian and binary $\mtp$ distributions}
\label{sec:examples}

Many examples of $\mtp$ random variables are discussed in the literature; see, e.g., \cite{KarlinRinott80,karlinGaussian}. In this section we focus on binary and multivariate Gaussian $\mtp$ distributions. Although the $\mtp$ property may appear restrictive, we want to suggest that $\mtp$ distributions are important in practice and in fact appear in real data sets. %We illustrate this with several examples.

\subsection{Multivariate Gaussian $\mtp$ distributions}\label{sec:gaussian}

Consider a multivariate Gaussian random vector $X$ with mean $\mu$ and covariance matrix $\Sigma$. Denote by $K$ the inverse of $\Sigma$. Then, the distribution of $X$ is $\mtp$ if and only if $K$ is an
M-matrix; see \cite{karlinGaussian}, i.e.,
\begin{itemize}
\item[(i)] $k_{vv}>0$ for all $v\in V$,
\item[(ii)] $k_{uv}\leq 0$ for all $u,v\in V$ with $u\neq v$.
\end{itemize}

Properties and consequences of M-matrices were studied by Ostrowski~\cite{O} who chose the name to honour H.~Minkowski who had considered aspects of such matrices earlier. The connection to multivariate Gaussian distributions was established by B{\o}lviken~\cite{B}.

In the previous section we mentioned that if $X$ is $\mtp$, then its constituent variables are associated. Therefore, for $\mtp$ Gaussian distributions it holds that $\sigma_{uv}\geq 0$ for all $u,v\in V$. More precisely, the covariance matrix has a block diagonal structure and each block has only strictly positive elements; see also Theorem \ref{th:block} below. Note, however, that this condition is necessary but not sufficient for the $\mtp$ property.

We now analyze by simulation how restrictive the $\mtp$ constraint is for Gaussian distributions. We quantify this by studying the ratio of the volume of all correlation matrices that satisfy the $\mtp$ constraint to  the volume of all correlation matrices. Since no closed-form formula for these volumes is known, we use a simple Monte Carlo simulation. We uniformly sample correlation matrices using the method suggested by Joe~\cite{Joe20062177}, which is implemented in the \texttt{R} package \texttt{clusterGeneration}. We performed simulations for $|V|=3,4,5$ and we here report how many correlation matrices out of $100{,}000$ samples satisfy the $\mtp$ constraint.
\begin{table}[htp!]
\begin{tabular}{c|ccc}
$|V|$ & 3 & 4 & 5\\
\hline
$\mtp$ & 5004 &  90 & 0
\end{tabular}
\end{table}

These simulation results show that the relative volume of $\mtp$ Gaussian distributions
drops dramatically with increasing~$|V|$ when no conditional independences are taken into
account. However, the picture changes when imposing conditional independence relations. For example, if $|V|=3$, then by the above simulations about $5\%$ of all Gaussian distributions correspond to $\mtp$ distributions. If $1\indep 2\mid 3$, then by a symmetry argument precisely $25\%$ of such distributions are $\mtp$. If, in addition, we impose $1\indep 3\cd 2$ --- which implies also $1\indep (3,2)$ --- the ratio of $\mtp$ distributions increases to $50\%$. Finally, \textit{all} distributions that are fully independent are $\mtp$.  

We next discuss a prominent data set consisting of the examination marks of $88$ students in five different mathematical subjects. The data were reported in \cite{mardia1979} and analyzed, for example, in \cite{edwards,GMsymms,whittaker1990}. The inverse of the sample covariance matrix, together with the corresponding partial correlations $\rho_{uv\cdot V\setminus \{u,v\}}$, are displayed in Table \ref{tab:concex1}. This matrix is very close to being an M-matrix with only one negative partial correlation equal to $-0.00001$. 
Furthermore, when fitting reasonable graphical models to the data, all fitted distributions are $\mtp$.

\begin{table}[htp!]
\caption{Empirical partial correlations (below the diagonal) and entries of the inverse of the sample covariance matrix ($\times 1000$, on and above the diagonal) for the examination marks in five mathematical subjects.}\label{tab:concex1}
\begin{tabular}{l|rrrrr}
 & Mechanics & Vectors & Algebra & Analysis &Statistics\\
\hline
Mechanics & $5.24$ & $-2.44$ & $-2.74$ & $0.01$ & $-0.14$\\
Vectors & 0.33& $10.43$ & $-4.71$ & $-0.79$ & $-0.17$\\
Algebra &  0.23&  0.28& $26.95$ & $-7.05$ & $-4.70$\\
Analysis &  -0.00&  0.08&  0.43& $9.88$ & $-2.02$\\
Statistics &  0.02&  0.02& 0.36 &  0.25& $6.45$\\
\end{tabular}
\end{table}

\subsection{Binary $\mtp$ distributions}\label{sec:binary}

Suppose that $X$ is a binary random vector with $\cX=\{0,1\}^{|V|}$ and we denote its distribution by $P=[p_{x}]$ for $x\in \cX$. For example, if $|V|=3$, binary $\mtp$ distributions must satisfy the following nine inequalities
\begin{eqnarray}\label{eq:3binary}
\nonumber p_{011}p_{000}\geq  p_{010}p_{001} & p_{101}p_{000}\geq  p_{100}p_{001}  & p_{110}p_{000}\geq  p_{100}p_{010} \\
p_{111}p_{100}\geq p_{110}p_{101} & p_{111}p_{010}\geq p_{110}p_{011}  & p_{111}p_{001}\geq  p_{101}p_{011} \\
\nonumber p_{111}p_{000}\geq  p_{100}p_{011} & p_{111}p_{000}\geq p_{010}p_{101}  & p_{111}p_{000}\geq p_{001}p_{110} .
\end{eqnarray}
The first two rows correspond to the inequalities $p_{x\wedge y}p_{x\vee y}\geq p_{x}p_{y}$ as in (\ref{eq:MTP2}), where $x$ and $y$ differ only in two entries. These inequalities are equivalent to requesting that the six possible conditional log-odds ratios are non-negative. By Proposition \ref{KarlinRinott80}, the inequalities in the last row are implied by the remaining ones in the case when $P>0$, or more generally, if $P$ has interval support. For $P>0$ this can be seen from identities of the form
$$
p_{111}p_{000}-p_{010}p_{101}\;=\;\frac{p_{000}}{p_{001}}(p_{111}p_{001}-p_{101}p_{011})\;+\;\frac{p_{101}}{p_{001}}(p_{011}p_{000}-p_{010}p_{001}).
$$
For positive binary distributions  we can verify $\mtp$ for any pair of variables with the remaining variables fixed. In the binary case this gives a single constraint for any choice of a pair and values for the remaining $|V|-2$ variables, in other words ${|V|\choose 2}\cdot 2^{|V|-2}$ inequalities. For binary $\mtp$ distributions there is a nice description in terms of log-linear parameters in~\cite{forcina2000}, see also Corollary \ref{cor:forcina} below.

The $\mtp$ hypothesis is rather restrictive in the binary setting when no further conditional independence restrictions are assumed. Note, however, that binary models can become more complex than in the Gaussian case, since  log-linear interactions of higher-order than pairwise may be present. In the following, we study the volume of $\mtp$ distributions with respect to the volume of the whole probability simplex. Similarly as in the Gaussian setting, we sample uniformly from the probability simplex. We here report how many samples out of 100,000 satisfy the $\mtp$ constraints for $|V|=3,4$. Note that already for $|V|=4$ we did not find a single instance although the volume of the set of $\mtp$ distributions is always positive:
\begin{table}[htp!]
\begin{tabular}{c|ccc}
$|V|$ & 3 & 4 \\
\hline
$\mtp$ & 2195 &  0
\end{tabular}
\end{table}

Like in the Gaussian case the relative volume of $\mtp$ distributions is higher when imposing additional conditional independence restrictions. By the same symmetry argument as in the Gaussian setting we obtain that for $|V|=3$ precisely $25\%$ of all binary distributions satisfying $1\indep 2 \cd 3$ are $\mtp$. If, in addition, we have $1\indep 3\cd 2$ then half of these distributions are $\mtp$. Finally, all binary full independence distributions are $\mtp$.

This interplay with conditional independence might explain in part why binary $\mtp$ distributions do arise in practice. See \cite[Section 5]{wermuth2014star} for examples of datasets that are $\mtp$ or nearly $\mtp$. In the following, we discuss two such examples.

\begin{ex}\rm We first consider a  dataset on \emph{EPH-gestosis}, collected 40 years ago in a study on ``Pregnancy and Child Development'' by the German Research Foundation and recently analyzed in~\cite[Section 5.1]{wermuth2014star}. EPH-gestosis represents a disease syndrome for pregnant women. The three symptoms are edema (high body water retention), proteinuria (high amounts of urinary proteins) and hypertension (elevated blood pressure). The observed counts $N=(n_{x})$~are
$$
\left[\begin{array}{cccc}
n_{000} & n_{010} & n_{001} & n_{011} \\
n_{100} & n_{110} & n_{101} & n_{111}
\end{array}\right]\quad=\quad \left[\begin{array}{cccc}
3299 & 107 & 1012 & 58 \\
78 & 11 & 65 & 19
\end{array}\right].
$$

If untreated, EPH-gestosis is a  major cause of death of mother and child during  birth~\cite[p.~65]{SZP}. However, treatment of the symptoms prevents negative consequences and the symptoms occur rarely after the first pregnancy. 

The observed counts have odds-ratios larger than one for each pair at the fixed level of the third variable, hence the empirical distribution is $\mtp$. Equivalently, the sample distribution satisfies all the constraints in (\ref{eq:3binary}). The three symptoms do not occur more frequently jointly than in pairs and the observed conditional odds-ratios are nearly equal given the third symptom. Possible  interpretations are that physicians intervened at the latest when two
symptoms occurred and that a single common cause, though unknown and unobserved, may have generated the marginal dependences between the symptom pairs. \qed
\end{ex}

\begin{ex}\rm\label{ex:laryngeal}
Next we discuss an example on five binary random variables.
This is a subset of data from a Polish case-control
study on laryngeal cancer \cite{Zatonski}. Details  on the study design, our selection criteria
for cases and controls, and the analysis will be given elsewhere.

In case-control studies  the observations are implicitly obtained conditionally on the values of at least  one response variable and
 on  relevant explanatory variables. For such designs,
the class of concentration graph models are appropriate for studying dependence structure among the variables.

In this study, we have 185  \emph{laryngeal cancer cases}   in urban residential areas (coded 1;  35.7\%) and 308 controls, coded 0. Four further  $0,1$ variables are defined so that  1 indicates the level known to carry  the higher cancer risk, namely \emph{heavy vodka drinking} (1:= regularly for 2 or more years; 21.3\%), \emph{heavy cigarette smoking} (1:= 30 or more cigarettes per day; 13.8\%, and 0:= 6 to 29 cigarettes per day), \emph{age at study entry} (1:= 54 to 65 years; 51.5\% and 0:= 46 to 53 years), and \emph{level of formal education} (1:= less than 8 years;  57.8 \% and  0:=8 to 11 years).

A well-fitting log-linear model for these data is determined by the sufficient margins $\{\{1,2\}, \{1,3\},\{2,3\},  \{1,4,5\}\}$, in other words by permitting log-linear interaction terms only among variable groups that are subsets of these sets.  This model yields  an overall likelihood-ratio $\chi^2$ of $13.6$ with 19 degrees of freedom and corresponds to a concentration graph with cliques: $\{1,2,3\}$ and $\{1,4,5\}$.
The  corresponding observed and  fitted counts are:
\begin{center}
{\scriptsize
$$
\left[\begin{array}{cccc}
00000 & 10000 & 01000 & 11000 \\
00100 & 10100 & 01100 & 11100 \\
00010 & 10010 & 01010 & 11010 \\
00110 & 10110 & 01110 & 11110 \\
00001 & 10001 & 01001 & 11001 \\
00101 & 10101 & 01101 & 11101 \\
00011 & 10011 & 01011 & 11011 \\
00111 & 10111 & 01111 & 11111 \\
\end{array}\right]: \left[\begin{array}{rrrr}
85 & 11  & 5 & 6 \\
10 & 1  & 1  & 2 \\
46 & 15 & 3 & 7 \\
7 & 2 & 2 & 5 \\
51&27&7&18\\
4&6 & 1& 4 \\
73 &36 & 5& 30 \\
5 & 9 & 3 & 6 \\
\end{array}\right],  \left[\begin{array}{rrrr}
85.88 & 9.87  & 7.30 & 6.35 \\
9.27 & 1.70  & 1.55  & 2.08 \\
47.59 & 14.31 & 4.19 & 9.21 \\
5.32 & 2.47 & 0.89 & 3.02 \\
51.70&27.13 &4.55&17.46\\
5.78 & 4.68 & 0.97& 5.73 \\
70.57 & 39.96 & 6.22 & 25.72 \\
7.89 & 6.89& 1.32 & 8.43 \\
\end{array}\right]\!\!.
$$}
\end{center}
For  pairs within the cliques,  the fitted two-way margins must coincide with the observed  bivariate tables of counts; here we report
marginal observed and fitted odds-ratios, $\oddsr(I,J)$, and  fitted conditional odds-ratios given the remaining variables, $\oddsr(I,J\cd R)$:
\begin{center}
{\small
\begin{tabular}{lcccccccccccc} \hline\\[-3mm]
variable pair:&1,2&1,3& 1,4& 1,5&2,3& 2,4 &2,5& 3,4& 3,5&4,5\\ \hline \\[-2mm]
observed $\oddsr(I,J)$: & 7.6 & 1.9 & 1.7  & 3.0 & 2.3 & 1.4& 2.0& 1.3& 0.9& 2.0\\[1mm]
fitted $\oddsr(I,J)$:&  & & & & &1.3& 1.6& 1.1& 1.2&  \\[1mm]
fitted $\oddsr(I,J\cd R)$:& 7.3 & 1.5& 2.5 & 4.4& 1.9&1&1 &1 &1 & ${}^{*)}$\\
\hline\\[-5mm]
\end{tabular}}\end{center}
{\quad \small ${}^{*)}$  $2.4$ for controls and $1.02$ for cases.\\[-4mm]}

Because the  observed $(3,5)$ odds-ratio is smaller than $1$ and hence the log-odds-ratio is negative,
the observed distribution is not $\mtp$.  In addition, $21$ of the  observed $80$ conditional log-odds ratios, $\oddsr(I,J\cd R)$, are less than $1$. However in the well-fitting model described above we have $\oddsr(I,J\cd R)\geq 1$ for  all $80$  odds-ratios,   so that the fitted distribution is $\mtp$.  This implies that each possible marginal table --- here of two, three, or four variables --- shows positive or vanishing pairwise dependence for all variable pairs.

From the concentration graph it follows that prediction of  drinking  and smoking habits cannot be
improved by using information about age or level of formal education for the studied cases or controls and that there is no  log-linear interaction involving more than three factors. The set of minimal sufficient tables tells that the only three-factor interaction is  in the $\{1,4,5\}$-table.  From the above change in the conditional odds-ratio for pair $(4,5)$ from $2.4$
to $1.02$, it follows that the expected improvement in education for younger -- compared to older participants -- only shows for controls but not for the cases. In combination with the fact that
$\oddsr(1,5\cd R)=4.4$, this implies that level of formal education should be explicitly included  in comparisons of  results across countries
 and  in  future  studies  on laryngeal cancer.
\end{ex}

\section{Conditional independence models and total positivity}
\label{sec:CIMTP}

An \emph{independence model} $\mathcal{J}$ over a finite set $V$ is a set of triples $\langle A,B\cd C\rangle$ (called \emph{independence statements}), where $A$, $B$, and $C$ are disjoint subsets of $V$; $C$
may be empty, and $\langle \varnothing,B\cd C\rangle$ and $\langle A,\varnothing\cd C\rangle$ are always included in $\mathcal{J}$. The independence statement $\langle A,B\cd C\rangle$ is read as ``$A$ is independent of $B$ given $C$''. Independence models do not necessarily have a  probabilistic interpretation; for a discussion on general independence models, see \cite{stu05}.

An independence model $\mathcal{J}$ over a set $V$ is a \emph{semi-graphoid} if it satisfies the following four properties for disjoint subsets $A$, $B$, $C$, and $D$ of $V$:
 \begin{enumerate}[({S}1)]
    \item $\langle A,B\cd C\rangle\in \mathcal{J}$ if and only if $\langle B,A\cd C\rangle\in \mathcal{J}$ (\emph{symmetry});
    \item if $\langle A,B\cup D\cd C\rangle\in \mathcal{J}$, then $\langle A,B\cd C\rangle\in \mathcal{J}$ and $\langle A,D\cd C\rangle\in \mathcal{J}$ (\emph{decomposition});
    \item if $\langle A,B\cup D\cd C\rangle\in \mathcal{J}$, then $\langle A,B\cd C\cup D\rangle\in \mathcal{J}$ and $\langle A,D\cd C\cup B\rangle\in \mathcal{J}$ (\emph{weak union});
    \item $\langle A,B\cd C\cup D\rangle\in \mathcal{J}$ and $\langle A,D\cd C\rangle\in \mathcal{J}$
    if and only if $\langle A,B\cup D\cd C\rangle\in \mathcal{J}$ (\emph{contraction}).
 \end{enumerate}
A semi-graphoid for which the reverse implication of the weak union property holds is said to be a \emph{graphoid} that is, it also satisfies
\begin{enumerate}[({S}1)]
\addtocounter{enumi}{4}
	\item if $\langle A,B\cd C\cup D\rangle\in \mathcal{J}$ and $\langle A,D\cd C\cup B\rangle\in \mathcal{J}$ then $\langle A,B\cup D\cd C\rangle\in \mathcal{J}$ (\emph{intersection}).
\end{enumerate}
Furthermore, a graphoid or semi-graphoid for which the reverse implication of the decomposition property holds is said to be \emph{compositional}, that is it also satisfies
\begin{enumerate}[({S}1)]
\addtocounter{enumi}{5}
	\item if $\langle A,B\cd C\rangle\in \mathcal{J}$ and $\langle A,D\cd C\rangle\in \mathcal{J}$ then $\langle A,B\cup D\cd C\rangle\in \mathcal{J}$ (\emph{composition}).
\end{enumerate}
Some independence models have additional properties;  below we write singleton sets $\{u\},\{v\}$ compactly as $u,v$, etc.
\begin{enumerate}[({S}1)]
\addtocounter{enumi}{6}
	\item if $\langle u,v\cd C\rangle\in \mathcal{J}$ and $\langle u,v\cd C \cup w\rangle\in \mathcal{J}$, then $\langle u,w\cd C\rangle\in \mathcal{J}$ or $\langle v,w\cd C\rangle\in \mathcal{J}$ (\emph{singleton-transitivity});

	\item if $\langle A,B\cd C\rangle\in \mathcal{J}$ and $D\subseteq V\setminus(A\cup B)$, then $\langle A,B\cd C\cup D\rangle\in \mathcal{J}$ (\emph{upward-stability}).
\end{enumerate}
The properties above are not independent and upward-stability is a very strong property. For example, we have the following simple lemma:
\begin{lem}\label{lem:axioms} Any upward stable semi-graphoid satisfies (S6) composition.
\end{lem}
\begin{proof}
If $\langle A,B\cd C\rangle\in \mathcal{J}$,  (S8) yields $\langle A,B\cd C\cup D\rangle\in \mathcal{J}$; hence from (S4) we get that $\langle A,D\cd C\rangle\in \mathcal{J}$ implies $\langle A,B\cup D\cd C\rangle\in \mathcal{J}$, which is (S6).
\end{proof}

A fundamental example of an independence model is induced by separation in an undirected graph $G=(V,E)$, denoted by $\mathcal{J}(G)$:
$$\langle A,B\cd S\rangle\in \mathcal{J(G)} \iff \mbox{ $S$ separates $A$ from $B$}$$ in the sense that all paths between $A$ and $B$ intersect $S$. The independence model  $\mathcal {J(G)}$  satisfies all of the above properties (S1)--(S8).

%\subsection{Conditional independence and total positivity}
Consider a set $V$ and associated random variables
$X=(X_v)_{v\in V}$. For disjoint subsets $A$, $B$, and $C$ of $V$
we use the short notation $A\cip B\cd C$ to denote that $X_A$ is \emph{conditionally independent of $X_B$ given $X_C$} \cite{daw79,lau96}, i.e.\ that for any measurable $\Omega\subseteq \mathcal{X}_A$ and $P$-almost all $x_B$ and $x_C$,
$$P(X_A \in \Omega\cd X_B=x_B, X_C=x_C)=P(X_A \in \Omega\cd X_C=x_C).$$
We can now induce an independence model $\mathcal{J}(P)$ by letting
\begin{displaymath}
\langle A,B\cd C\rangle\in \mathcal{J}(P) \text{ if and only if } A\cip B\cd C \text{ w.r.t.\ $P$}.
\end{displaymath}

Probabilistic independence models are always semi-graphoids~\cite{Pearl:87}. If, for example, $P$ has a strictly positive density $f$, the induced independence model is always a graphoid; see e.g.\ Proposition~3.1 in \cite{lau96}. More generally, if $f$ is continuous, Peters~\cite{peters:15} showed that the induced independence model is a graphoid if and only if the support is \emph{coordinate-wise connected}, i.e., all connected components of the support of the density can be connected by axis-parallel lines. In particular, this applies to the discrete case since any function over a discrete space is continuous. See also \cite{dawid:80} for  general discussions  and \cite{sanmartin2005} for necessary and sufficient conditions under which the intersection property  holds in joint Gaussian or binary distributions.

Examples of discrete distributions
violating one of (S5), (S6), or (S7) have been given in~\cite{W12}. We will prove in this section that, under weak assumptions,  independence models generated by $\mtp$ distributions satisfy all of the properties (S1)-(S8).

First, note that by Proposition~\ref{prop:ClosedMC} the $\mtp$ property is closed under marginalization and conditioning. Applying this to the conditional distribution of $(X_u,X_v)$ given $X_C$ for all $u,v\in V$ with $u\neq v$, $C\subseteq V\setminus\{u,v\}$, implies that the following conditional covariances must be nonnegative
\begin{equation}
\label{eq_conditioning}
{\rm cov}\{\phi(X_{u}),\psi(X_{v})\cd X_{C}\}\geq 0 \qquad \textrm{a.s.}
\end{equation}
for non-decreasing functions $\phi,\psi:\R\to \R$ for which the covariance exists. Recall that a function of several variables is non-decreasing if it is non-decreasing in each coordinate. The following related result was first proved in \cite[Section 3.1]{sarkar} (see also \cite[Theorem~4.1]{KarlinRinott80}).

\begin{prop}
\label{thm_KarlinRinot}
Let $X$ be $\mtp$. Then for any subset $A\subseteq V$ and any non-decreasing function $\varphi:\cX_A\to \R$ for which $\E|\varphi(X_A)|<\infty$, the conditional expectation
$$
\E\left\{\varphi(X_A)\cd X_{V\setminus A}=x_{V\setminus A}\right\}
$$
is non-decreasing in $x_{V\setminus A}$.
\end{prop}

We first show that any induced independence model of an $\mtp$ distribution is an upward-stable  and singleton-transitive compositional semigraphoid, i.e.\ (S1)--(S4) and (S6)--(S8) all hold.
\begin{thm}
\label{prop:intersection}
Any independence model $\mathcal{J}(P)$ induced by an $\mtp$ distribution $P$ is an upward-stable and singleton-transitive compositional semigraphoid.\end{thm}

\begin{proof}
We first note that any probabilistic independence model is a semi-graphoid~\cite{pea88}. Next,  we establish upward-stability. For this, it suffices to prove that ${u}\cip {v}\cd {C}$ implies ${u}\cip {v}\cd {C\cup \{w\}}$ for all $w\in V\setminus (C\cup \{u,v\})$. Since the $\mtp$ property is closed under marginalization, it follows that the marginal distribution of $X_{C\cup \{u,v,w\}}$ is $\mtp$. Further, because the $\mtp$ property is closed under conditioning, after conditioning on $C$, it suffices to consider only 3 variables and prove the following statement: If the distribution of $X=(X_{1},X_{2},X_{3})$ is $\mtp$, then  ${1}\cip {2}$ implies ${1}\cip {2}\cd {3}$.

Recall that $1\cip 2$ if and only if $\P(X_1> a_1,X_2> a_2)=\P(X_1> a_1)\P(X_2> a_2)$ for all $a_1,a_2\in \R$. Equivalently, defining $Y_i=\indic(X_i> a_i)$ this translates into $Y_1\cip Y_2$, which now must be satisfied for any choice of $a_1,a_2\in \R$. A similar condition holds for the conditional independence $1\cip 2 \cd 3$, in which case we require that $Y_1\cip Y_2\cd X_3$ for all $a_1,a_2\in \R$.

The advantage of working with the indicator functions is that they define bounded random variables, which implies the existence of moments. By Proposition \ref{prop:monotone} the vector $(Y_1,Y_2,X_3)$ is $\mtp$. Independence of $X_1$ and $X_2$, and the law of total covariance,  implies that for every choice of $a_1,a_2\in \R$
\begin{eqnarray*}
0 &=& \cov(Y_1, Y_2)\\
&=& \cov(\mathbb{E}(Y_1\cd X_3), \mathbb{E}(Y_2\cd X_3)) + \mathbb{E}(\cov(Y_1,Y_2\cd X_3)).
\end{eqnarray*}
By Proposition~\ref{thm_KarlinRinot}, $\mathbb{E}(Y_1\cd X_3=x_3)$ and $\mathbb{E}(Y_2\cd X_3=x_3)$ are almost everywhere non-decreasing and bounded functions of $x_3$ and hence their covariance exists and is non-negative by (\ref{eq_increasing}) and the fact that univariate random variables are always associated. Moreover, it follows from (\ref{eq_conditioning}) that $\cov(Y_1,Y_2\cd X_3=x_3)\geq 0$ for almost all $x_3$, and thus its expectation is non-negative. This means that we expressed zero as a sum of two non-negative terms and thus both terms must be zero. This implies that $\cov(Y_1,Y_2\cd X_3=x_3)= 0$ for almost all $x_3$. Hence, by Theorem~\ref{th:newman}, we obtain that $Y_1\cip Y_2\cd X_3$. Now by varying $a_1,a_2$ we conclude that $X_1\indep X_2\cd X_3$.

Having established upward-compatibility, composition now follows from Lemma~\ref{lem:axioms}. 

We finally prove singleton-transitivity. Using upward-stability this property can  be rephrased in a simpler form as
$$1\cip 2 \quad\Longrightarrow\quad 1\cip 3\quad\textrm{or}\quad 2\cip 3.$$
So we assume that $1\cip 2$. By the same argument as in the previous paragraph, ${\rm cov}(Y_1,Y_2)=0$ implies that  $\cov(\mathbb{E}(Y_1\cd X_{3}), \mathbb{E}(Y_2\cd X_{3}))=0$. By Theorem~\ref{thm_KarlinRinot}, both $f(X_3):=\mathbb{E}(Y_1\cd X_{3})$ and $g(X_3):=\mathbb{E}(Y_2\cd X_{3})$ are   non-decreasing functions. Their covariance is zero and can be rewritten as
$$
{\rm cov}(f(X_3),g(X_3))\;=\;\int_{\{x'>x\}}\big(f(x')-f(x)\big)\big(g(x')-g(x)\big)\,d(\mu_3(x)\otimes \mu_3(x')).
$$
Note that $\{x'>x\}\supseteq \{f(x')>f(x)\}\cap \{g(x')>g(x)\}$ and so the integral on the right-hand side, which is equal to zero, is bounded from below by $$
\int_{\{f(x')>f(x)\}\cap \{g(x')>g(x)\}}\big(f(x')-f(x)\big)\big(g(x')-g(x)\big)\,d(\mu_3(x)\otimes \mu_3(x')),
$$
which is strictly positive unless the set $\{f(x')>f(x)\}\cap \{g(x')>g(x)\}$ has measure zero. This set is the set of all pairs $(x,x')$ such that $x'>x$ and $f(x')> f(x)$, $g(x')> g(x)$, and its measure is half the measure of the set of all $(x,x')$ such that $f(x')\neq f(x)$ and $g(x')\neq g(x)$. This measure is zero only if either $f$ or $g$ is constant almost everywhere.
 
It follows that for every $a_1,a_2\in \R$ either $\,\mathbb{E}(Y_1\cd X_3)=\P(X_1> a_1)\,$ or $\,\mathbb{E}(Y_2\cd X_3)=\P(X_2> a_2)\,$. 
Let $U\subseteq \R^2$ be the set of all $(a_1,a_2)$ such that the former equality holds and $V$ be the set such that the latter holds.  Let $\pi_i$ denote the projection on the $i$-th coordinate in $\R^2$. We have $U\cup V=\R^2$ and so if $\pi_1(U)\neq \R$ then $\pi_2(V)=\R$. This implies that $\pi_1(U)=\R$ or $\pi_2(V)=\R$. For simplicity assume that the latter holds but the general argument is the same. If $\pi_2(V)=\R$ then for every $a_2\in \R$ there exists $a_1$ such that $\,\mathbb{E}(Y_2\cd X_3)=\P(X_2> a_2)\,$, or equivalently $\P(X_2>a_2|X_3)=\P(X_2>a_2)$. We conclude that $2\indep 3$. This, up to symmetry, implies  that $1\cip 3$ or $2\cip 3$.
\end{proof}

We now analyze the intersection property.  It is important to note that an $\mtp$ independence model is not necessarily a graphoid, as the following simple example shows.

\begin{ex}\rm\label{ex:intcounex}
Consider the binary $\mtp$ distribution with
$$p_{000}=p_{111}=\frac{1}{2}.$$
Then $1\cip 2\cd 3$ and $1\cip 3\cd 2$, but $1\notindependent (2,3)$, and therefore, the intersection property does not hold. \qed
\end{ex}

As a consequence of the earlier mentioned result by Peters \cite{peters:15}, any $\mtp$ distribution with continuous density and coordinate-wise connected support is an upward-stable and singleton-transitive compositional graphoid.

We conclude this section with the following property of $\mtp$ distributions.

\begin{thm}\label{th:block}
Let the distribution of $X$ be $\mtp$ with  none of $X_v$ having a degenerate distribution. Then $X$ can be decomposed into independent components such that within each component all variables are mutually marginally dependent.
\end{thm}

\begin{proof} As in the previous proof we define
$Y_i=\indic(X_i> a_i)$
and by Theorem~\ref{th:newman} it suffices to prove that the covariance matrix of $Y$ is block diagonal with strictly positive entries in each block.
We write $u\sim v$ if the covariance between $Y_{u}$ and $Y_{v}$ is non-zero and we show that $u\sim v$ is an equivalence relation and thus induces a partition of $V$ into independent blocks. It is clear that $u\sim u$ and $u\sim v$ whenever $v\sim u$. It remains to show that $u\sim v$ and $v\sim w$ imply $u\sim w$. But if $u\not\sim w$, we have $\sigma_{uw}=0$ and thus $u\cip w$. Using upward-stability from Theorem~\ref{prop:intersection} yields $u\cip w\cd v$ and singleton-transitivity yields $u\cip v$ or $v\cip  w$,
which contradicts that $u\sim v$ and $v\sim w$.
\end{proof}

\section{Faithfulness and total positivity}
\label{sec:Faithful}

In the following we shall write $A\gse B\cd C$ for the graph separation $\langle A,B\cd C\rangle\in \mathcal{J(G)}$ and $A\cip B\cd C$ for the relation $\langle A,B\cd C\rangle\in \mathcal{J(P)}$ in the independence model generated by~$P$. For a graph $G=(V,E)$, an independence model $\mathcal{J}$ defined over $V$ satisfies the \emph{global Markov property} w.r.t.\ a graph $G$, if for  disjoint subsets $A$, $B$, and $C$ of $V$ the following holds  $$A\gse B\cd C \implies \langle A,B\cd C\rangle\in \mathcal{J}.$$
If $\mathcal{J}(P)$ satisfies the global Markov property w.r.t.\ a graph $G$, we also say that \emph{$P$ is Markov} w.r.t.\ $G$.

We say that an independence model $\mathcal{J}$ is \emph{probabilistic} if there is a distribution $P$ such that  $\mathcal{J}= \mathcal{J}(P)$. We then also say that $P$ is \emph{faithful} to $\mathcal{J}$. If $P$ is faithful to $\mathcal{J}(G)$ for a graph $G$, then we also say that $P$ is \emph{faithful to  $G$}. Thus, if $P$ is faithful to $G$ it is also Markov w.r.t.\ $G$.

In this section we examine the faithfulness property for $\mtp$ distributions. Let $P$ denote a distribution on $\cX$.  The \emph{pairwise independence graph}  of $P$ is the undirected graph $G(P)=(V, E(P))$ with
$$uv\notin E(P) \quad\Longleftrightarrow \quad u\cip v\cd {V\setminus\{u,v\}}.$$
A distribution $P$ is said to satisfy the \emph{pairwise Markov property} w.r.t.~an undirected graph $G=(V,E)$ if
$$uv\notin E\quad \Longrightarrow \quad u\cip v\cd {V\setminus\{u,v\}}.$$
Thus any distribution $P$ satisfies the pairwise Markov property w.r.t.~its pairwise independence graph $G(P)$; indeed, $G(P)$ is the smallest graph that makes $P$  pairwise Markov.

Generally, a distribution may be pairwise Markov w.r.t.\ a graph without being globally Markov.
However, if an $\mtp$ distribution $P$ satisfies the coordinate-wise connected support condition, in particular if it is strictly positive, and since these are sufficient conditions for the intersection property to hold, then pairwise and global Markov properties are equivalent; see \cite{sadl14}. We prove in the following result that if the pairwise-global equivalence is already established, then $P$ is in fact faithful to $G(P)$.

\begin{thm}
\label{th:Faith}
If $P$ is $\mtp$ and its independence model is a graphoid, then $P$ is faithful to its pairwise independence graph $G(P)$.
\end{thm}
\begin{proof}
By Theorem~3.7 in~\cite{lau96} it follows that $P$ is globally Markov w.r.t.\ $G(P)$.

To establish faithfulness, we consider disjoint subsets $A$, $B$, and $C$ so that $C$ does not separate $A$ from $B$ in $G(P)$. We need to show that $A\notindependent B\cd C$.

First, let $uv\in E$. Then $u\notindependent v\cd {V\setminus\{u,v\}}$ and hence by upward-stability as shown in Theorem~\ref{prop:intersection}, $u\notindependent v\cd C$ for any $C\subset V\setminus\{u,v\}$.

Since $C$ does not separate $A$ from $B$, there exists $u\in A$ and $v\in B$ and a path $u=v_1, v_2,\dots , v_r = v$ such that $v_k\notin C$ for all $k=1,\dots, r$, and $v_k v_{k+1}\in E$ for all $k=1,\dots, r-1$. By the previous argument we obtain that ${v_k}\notindependent {v_{k+1}}\cd C$ for all $k=1,\dots, r-1$. By singleton-transitivity, ${v_1}\notindependent {v_{2}}\cd C$ and ${v_2}\notindependent {v_{3}}\cd C$ imply that ${v_1}\notindependent {v_3}\cd C$. Repeating this argument yields $u\notindependent v\cd C$ and hence $A\notindependent B\cd C$.
\end{proof}
Note that this was also shown by Slawski and Hein~\cite{slawski2015estimation} in the case of a Gaussian $\mtp$ distribution. In fact in the Gaussian case it follows readily since conditional covariances between any pair of variables can be obtained through adding (non-negative) partial correlations along paths in $G(P)$ \citep{jones:west:92}, see also \cite{WerCox98}.

Notice that if $X$ has coordinate-wise connected support, then Theorem~\ref{th:block} is a direct corollary of Theorem~\ref{th:Faith}. This is because if a distribution is faithful to an undirected graph, then the statement of Theorem \ref{th:block} obviously holds.  However, the latter theorem is still interesting as it covers cases that Theorem \ref{th:Faith} does not; such as that of Example~\ref{ex:intcounex}.

In Section~\ref{sec:MTP} we postponed to show that the stability of the $\mtp$ property under coarsening as established in (iii) of Proposition~\ref{prop:ClosedMC} does not imply that coarsening preserves conditional independence relations for $\mtp$ distributions. We demonstrate this in the following example.

\begin{ex}\rm\label{ex:distortion} Consider the trivariate discrete distributions of $(I,J,K)$ where $I$ and $K$ are binary taking values in $\{0,1\}$ whereas $J$ is ternary with state space $\{0,1,2\}$ given as follows
\[p_{ijk}=\theta_{i|j}\phi_{k|j}\psi_j,\]
where $\psi_j=1/3$ for all $j$, $\theta_{1|0}= \phi_{1|0}=1-\theta_{0|0}=1-\phi_{0|0}=1/4$, $\theta_{1|1}= \phi_{1|1}=1-\theta_{0|1}=1-\phi_{0|1}=1/3$, and $\theta_{1|2}= \phi_{1|2}=1-\theta_{0|2}=1-\phi_{0|2}=1/2$. This distribution is easily seen to be $\mtp$ which also follows from Proposition~\ref{prop:alongsingleton} below. By construction, it also satisfies $I\indep K\cd J$ so that its concentration graph has edges $IJ$ and $JK$.

Now define the binary variable $L$ by monotone coarsening of $J$ so that $L=0$ if $J=0$ and $L=1$ if $J\in \{1,2\}$. Letting $q_{ilk}$ denote the joint distribution of $(I,L,K)$ we get for example
\[q_{010}=p_{010}+p_{020}=(\theta_{0|1}\phi_{0|1}+\theta_{0|2}\phi_{0|2})/3=(4/9+1/4)/3=25/108,\]
and similarly
\[q_{011}=q_{110}=17/108,\quad q_{111}=13/108\] so that the odds-ratio between $I$ and $J$ conditional on $\{L=1\}$ becomes
\[\theta=\frac{q_{010}q_{111}}{q_{110}q_{011}}=\frac{13\times 25}{17^2}=\frac{325}{289}>1.\]
Hence, after coarsening, the conditional association between $I$ and $J$ given the third variable changes from absent to positive. Note that the $\mtp$ property ensures non-negativity of the distorted association.

Clearly, the distribution after coarsening remains faithful to its concentration graph, but coarsening changes the latter to become the complete graph on $V=\{I,L,K\}$. \qed
\end{ex}

For completeness of Theorem \ref{th:Faith} it is important to show that any concentration graph is realizable by an $\mtp$ distribution. We prove this fact in the special case of Gaussian distributions.
\begin{prop}
Any undirected graph $G$ is realizable as the concentration graph $G(P)$ of some $\mtp$ Gaussian distribution.
\end{prop}
\begin{proof}
Let $A$ be the adjacency matrix of $G$, that is, $A_{ij}=1$ if and only if $(i,j)$ is an edge of $G$. Because $G$ is undirected, A is symmetric. Since the set of positive definite matrices is open and the identity matrix is positive definite, it follows that $K=I-\epsilon A$ is positive definite if $\epsilon>0$ is sufficiently small. By construction, $K$  is an M-matrix and its non-zero elements correspond to the edges of the graph $G$.	
\end{proof}

\section{Special instances of total positivity}
\label{sec:factorization}

We conclude this paper with a section on how to construct $\mtp$ distributions from a collection of smaller $\mtp$ distributions, a brief discussion of conditions for the $\mtp$ property of discrete distributions in terms of log-linear interaction parameters, and characterizing conditional Gaussian distributions which are $\mtp$.

\subsection{Singleton separators}

Let $A,B\subset V$. We then say that two random variables $X_A$ and $X_B$ with distributions $P_A$ and $P_B$ are \emph{consistent} if the distribution of $X_{A\cap B}$ is the same under $P_A$ as under $P_B$. Then one can define a new distribution denoted by $P_A\star P_B$ and known as the \emph{Markov combination} of $P_A$ with density $f$ and  $P_B$ with density  $g$ (see~\cite{dawid:lauritzen:93}). Its density is denoted by $f\star g$ and given by
$$ (f\star g)(x_{A\cup B})=\frac{f(x_A)g(x_B)}{h(x_{A\cap B})}.$$ Here, $h$ denotes the density of $X_{A\cap B}$, common to $P_A$ and $P_B$. In the following, we show that the Markov combination of two $\mtp$ distributions is again $\mtp$ as long as they are glued together over a 1-dimensional margin.

\begin{prop}\label{prop:alongsingleton} Suppose that $|A\cap B|=1$. Then the Markov combination $P_A\star P_B$ of a consistent pair of distributions $P_A$ and $P_B$  is $\mtp$ if and only if $P_A$ and $P_B$ are both $\mtp$.
\end{prop}

\begin{proof}
Since $P_A$ and $P_B$ are marginal distributions of $P_A\star P_B$ and the $\mtp$ condition is preserved under marginalization, we only need to prove one direction. Assume that $P_A$ and $P_B$ are $\mtp$. The product of $\mtp$ functions is an $\mtp$ function; see, e.g. Proposition 3.3 in \cite{KarlinRinott80} for this basic result. This implies that  $(fg)(x)=f(x_A)g(x_B)$ is $\mtp$. Now, if $A\cap B$ is a singleton, then also $f\star g$ is $\mtp$ because multiplying by functions of a single variable preserves the $\mtp$ property; c.f. (\ref{eq:prodtrans}).
\end{proof}

For example,  Proposition~\ref{prop:alongsingleton} implies that for the fitted model in Example~\ref{ex:laryngeal} we only need to check the $\mtp$ condition in each of the two clique marginals $\{1,2,3\}$ and $\{1,4,5\}$ to verify that the fitted distribution is $\mtp$. Since the fitted distribution is positive, this involves only $6+6=12$ log-odds-ratios, see the discussion of (\ref{eq:3binary}) in Section~\ref{sec:binary}.  In addition, as there are only pairwise interactions in the $\{1,2,3\}$-marginal, conditional log-odds-ratios for any pair of these variables are constant in the third variable and hence we actually only need to check $3+6=9$ such ratios to verify the $\mtp$ property for the model fitted to the laryngeal cancer data; see also Theorem~\ref{prop:ifpsimtp} below.

Unfortunately, the conclusion in Proposition~\ref{prop:alongsingleton} does not hold in general if $|A\cap B|>1$, as we show in the following example.

\begin{ex}\rm\label{ex:simpleonlysingl}
Suppose that $A=\{1,2,3\}$ and $B=\{2,3,4\}$, and let $X=(X_1,X_2,X_3,X_4)\in \{0,1\}^4$. Consider the following distribution:
{\small\begin{align*}
[p_{0000}, p_{0001}, p_{0010}, p_{0011}, p_{0100}, p_{0101}, p_{0110}, p_{0111}]&=[1,2,2,20,2,20,20,400]/Z,\\
[p_{1000}, p_{1001}, p_{1010}, p_{1011}, p_{1100}, p_{1101}, p_{1110}, p_{1111}]&=[2,4,20,200,20,200,400,8000]/Z,
\end{align*}}%
where the normalizing constant $Z=9313$. It is easy to check that for every $i,j,k,l\in\{0,1\}$ the following holds
$$
p_{ijkl}\;=\;\frac{p_{ijk+}p_{+jkl}}{p_{+jk+}}.
$$
Hence, the distribution $P=[p_{ijkl}]$ can be obtained as the Markov combination of two distributions, namely $p_{ijk+}$ over $\{1,2,3\}$ and $p_{+jkl}$ over $\{2,3,4\}$. One can also easily check that both these distributions are $\mtp$. However, since
\begin{eqnarray*}
p_{(1,1,0,1)\wedge (1,0,1,1)}p_{(1,1,0,1)\vee (1,0,1,1)} - p_{(1,1,0,1)}p_{(1,0,1,1)} &=& \\
p_{1001}p_{1111}-p_{1101}p_{1011}&=&-\frac{8000}{9313^2},
\end{eqnarray*}
$P$ is not $\mtp$. \qed
\end{ex}

As a direct consequence of Proposition \ref{prop:alongsingleton} we obtain the following result for \emph{decomposable graphs}, which are graphs where there is no cycle of length more than three such that all its non-neighboring nodes  (on the cycle) are not adjacent (see e.g.~\cite{lau96} for a review).

\begin{cor}\label{th:deconly} Let $G$ be a decomposable graph such that the intersection of any two cliques is either empty or a singleton. Let $P$ be a distribution that is Markov w.r.t.~$G$. Then $P$  is $\mtp$ if and only if the marginal distribution  over each clique is $\mtp$.
\end{cor}
\begin{proof}
The proof follows by induction over the number of cliques.
\end{proof}

As we show in the following example, Corollary~\ref{th:deconly} cannot be extended directly to non-decomposable graphs. It does not hold in general that a distribution is $\mtp$ if the margins over all cliques in the graph are $\mtp$ and the cliques intersect in singletons only. However, as we show in Theorem~\ref{prop:ifpsimtp} such a result does hold if all clique potentials are $\mtp$ functions.

\begin{ex}\rm
Consider the following 4-dimensional binary distribution $P=[p_{ijkl}]$ with
$$
p_{ijkl}\;=\;\frac{1}{Z}\,A_{ij}B_{jk}C_{kl}D_{il},
$$
where
$$
Z=243, \qquad A=\begin{bmatrix}
6 & 5\\
4 & 3
\end{bmatrix}\qquad \textrm{and} \qquad B=C=D=\begin{bmatrix}
2 & 1\\
1 & 2
\end{bmatrix}.
$$
This distribution is Markov w.r.t.~the 4-cycle. We now show that the marginal distributions over each edge are $\mtp$. For this, note that a binary 2-dimensional random vector is $\mtp$ if and only if its covariance is non-negative. To see this, observe that in the bivariate case there is only one inequality $p_{00}p_{11}-p_{01}p_{10}\geq 0$. Using the fact that $p_{00}+p_{01}+p_{10}+p_{11}=1$, it is seen that this inequality is equivalent to $p_{11}-p_{1+}p_{+1}\geq 0$; but ${\rm cov}(X_1,X_2)=p_{11}-p_{1+}p_{+1}$.

In this example,
$$
{\rm cov}(X_1,X_2)=\frac{148}{243^2},\quad {\rm cov}(X_2,X_3)=\frac{4812}{243^2},\,
$$
$$
{\rm cov}(X_3,X_4)=\frac{4842}{243^2},\quad {\rm cov}(X_1,X_4)=\frac{4632}{243^2},
$$
and hence all edge-marginals are $\mtp$. However, the full distribution $P$ is not $\mtp$, since
$$
p_{0011}p_{1111}-p_{0111}p_{1011}=-\frac{32}{243^2},
$$
which completes the proof by a similar argument as in Example~\ref{ex:simpleonlysingl}. \qed
\end{ex}

We now show how to overcome these limitations and build $\mtp$ distributions over non-decomposable graphs, namely by using $\mtp$ potentials over the edges instead of $\mtp$ marginal  distributions over the edges.

\begin{thm}\label{prop:ifpsimtp}
A distribution of the form
$$
p(x)=\frac{1}{Z}\prod_{uv\in E} \psi_{uv}(x_u,x_v),
$$
where $\psi_{uv}$ are positive functions and $Z$ is a normalizing constant, is $\mtp$ if and only if each $\psi_{uv }$ is an $\mtp$ function.
\end{thm}

\begin{proof}
Since the distribution $p$ is strictly positive, by Proposition~\ref{KarlinRinott80} $p$ satisfies $\mtp$ if and only if it does so for $x,y\in \cX$ that differ in two coordinates, say with indices $u,v$. Write $E_u$ for the set of edges that contain $u$ but not $v$ and  $E_v$ for the set of edges that contain $v$ but not $u$. First, consider the case where $uv\in E$. Then we have that $p(x\wedge y)p(x\vee y)-p(x)p(y)\geq 0$ if and only if
\begin{multline*}
\psi_{uv}((x\wedge y)_{uv})\,\psi_{uv}((x\vee y)_{uv})\prod_{st\in E_u\cup E_v}\psi_{st}((x\wedge y)_{st})\,\psi_{st}((x\vee y)_{st})\geq \\
\psi_{uv}(x_{uv})\,\psi_{uv}(y_{uv})\prod_{st\in E_u\cup E_v}\psi_{st}(x_{st})\,\psi_{st}(y_{st}).
\end{multline*}
All other terms cancel out because of the assumption that $x_w = y_w$ for all $w\in V\setminus\{u,v\}$. Now note that for $st\in E_u\cup E_v$ we have $\{x_{st},y_{st}\}=\{(x\wedge y)_{st},(x\vee y)_{st}\}$ and so the above inequality holds if and only if
$$
\psi_{uv}((x\wedge y)_{uv})\psi_{uv}((x\vee y)_{uv})\geq \psi_{uv}(x_{uv})\psi_{uv}(y_{uv}),
$$
that is, if and only if $\psi_{uv}$ is $\mtp$.

Next consider the case where $uv\notin E$. By the same argument one concludes that in this case the above inequalities are in fact equalities, which completes the proof.
\end{proof}

As a final remark note that Theorem~\ref{prop:ifpsimtp}
 can directly be extended to distributions of the form
$$
p(x)=\frac{1}{Z}\prod_{C\in \mathcal{C}}\psi_C(x_C),
$$
where $\mathcal C$ is a family of subsets of $V$ such that for any two $C,C'\in \mathcal C$ we have $|C\cap C'|\in\{0,1\}$.

\subsection{Log-linear interactions}\label{sec:loglinear}

We next give a short discussion of interaction representations for discrete $\mtp$ distributions, as they typically are used in log-linear models for contingency tables. Suppose that $X=(X_{v})_{v\in V}$ is a random vector with values in $\cX=\prod_{v\in V}\cX_{v}$ where each $\cX_v$ is finite.
Let $\mathcal{D}$ denote the set of subsets of $V$. Any function $h:\cX\to \R^n$ of $\cX$ can be expanded as
\begin{equation}\label{eq:loglinear}
h(x)\;\;=\;\; \sum_{D\in \mathcal D}\theta_{D}(x),
\end{equation}
where $\theta_D$ are functions on $\mathcal X$ that only depend on $x$ through $x_D$, i.e.\ $\theta_D$ satisfy that $\theta_{D}(x)=\theta_{D}(x_{D})$.
In the case where $h(x)=\log p(x)$ where $p$ is a positive probability distributions over $\cX$, the functions
 $\theta_{D}(x)$ are known as \emph{interactions} among variables in $D$ and we shall also use this expression for a general function $h$.

Without loss of generality we may assume that $\min \cX_v=0$ for all $v\in V$ and to assure that the representation is unique, we may require that $\theta_{D}(x)=0$ whenever $x_{d}=0$ for some $d\in D$. With this convention, the sum in (\ref{eq:loglinear}) can be rewritten so it only extends over such  $D\in \mathcal D$ which are contained in the \emph{support} $S(x)$ of $x$ where $d\in S(x) \iff x_d\neq 0$. In the binary case, when $d_{v}=1$, this allows us to use a simpler notation, namely  $\theta_{D}(\mathbf 1_{D}):=\theta_{D}$ for all $D\in \mathcal D$.

For any such interaction expansion and a fixed pair $u,w\in V$, we define a function $\gamma_{uw}$ on $ \cX$ by
$$
\gamma_{uw}(x)=%\;\;:=\;\;\begin{cases}
\sum_{D:\{u,w\}\subseteq D\subseteq S(x)} \theta_{D}(x).% & \mbox{if }u,w\in S(x),\\
%0 & \mbox{otherwise}.
%\end{cases}
$$
Observe that then $\gamma_{uw}(x)=0$ unless $u,w\in S(x)$ and thus in particular whenever $|S(x)|\leq 1$; further,  $\gamma_{uw}$ is a linear combination of interaction terms.

Let $\cZ\subseteq \cX$ be a subset of $\cX$ which is closed under $\wedge$ and $\vee$. A function $g:\cZ\to \R$ is \emph{supermodular} if
$$
g(x\wedge y)+g(x\vee y)\;\;\geq \;\;g(x)+g(y)\quad\mbox{for all }x,y\in \cZ.
$$ Thus $g$ is supermodular on $\cX$ if and only if $\exp(g)$ is $\mtp$ on $\cX$.  A function $g$ is \emph{modular} if  both of $g$ and $-g$ are supermodular.

Denote by $\cX^{A}$ the set of all $x\in \cX$ with $S(x)=A$. Clearly, $\cX^A$ is closed under $\wedge$ and $\vee$.
Then we obtain the following result.

\begin{thm}\label{thm:fuwmtp}Let $P$ be a strictly positive distribution of $X$. Then $P$ is $\mtp$ if and only if  for all $A\subseteq V$ with $|A|\geq 2$ and any given $u,w\in V$ the function $\gamma_{uw}$ is  non-negative,  non-decreasing, and supermodular over  $\cX^{A}$.  \end{thm}
\begin{proof}By Proposition~\ref{KarlinRinott80} $p$ is $\mtp$ if and only if
\begin{equation}\label{eq:logsuper}
\log p(x\wedge y)+\log p(x\vee y)-\log p(x)-\log p(y)\geq 0
\end{equation}
for all $x,y\in \mathcal X$ that differ only in two entries. Let $u,w\in V$ and take $x,y\in \mathcal X$ satisfying $x_{v}=y_{v}$ for all  $v\in V\setminus\{u,w\}$. Without loss of generality we can assume $x_{u}<y_{u}$ and $y_{w}<x_{w}$ for otherwise the inequality is trivially satisfied. Using the expansion (\ref{eq:loglinear}), the inequality (\ref{eq:logsuper}) becomes %$\log p(x\wedge y)+\log p(x\vee y)-\log p(x)-\log p(y) \geq 0$ becomes
\begin{equation}\label{eq:sumthetas}
\sum_{D\in \mathcal D}\big(\theta_{D}((x\wedge y)_{D})+\theta_{D}((x\vee y)_{D})-\theta_{D}(x_{D})-\theta_{D}(y_{D})\big)\;\;\geq\;\;0.
\end{equation}
%where $\mathcal D$ is the set of all non-empty subsets of $V$.

For every $D\subseteq V\setminus \{w\}$ we have $(x\wedge y)_{D}=x_{D}$ and $(x\vee y)_{D}=y_{D}$. Similarly,  for every $D\subseteq V\setminus \{u\}$ we have $(x\wedge y)_{D}=y_{D}$ and $(x\vee y)_{D}=x_{D}$,  and thus, in both cases the corresponding summands in (\ref{eq:sumthetas}) are zero. It follows that  (\ref{eq:sumthetas}) is equivalent to
\begin{equation}\label{eq:sumthetas2}
\sum_{D:\{u,w\}\subseteq D\subseteq A}\big(\theta_{D}((x\wedge y)_{D})+\theta_{D}((x\vee y)_{D})-\theta_{D}(x_{D})-\theta_{D}(y_{D})\big)\;\;\geq\;\;0,
\end{equation}
where $A\subseteq V$ is the support of $x\vee y$ (if $D$ is not contained in $A$ all terms $\theta_{D}$ are zero by our convention).

We now show that the fact that (\ref{eq:sumthetas2}) must hold for all $u,w\in V$ and $x,y\in \mathcal X$ as above is equivalent to the fact that all $\gamma_{uw}$ satisfy the conditions of the theorem.

Consider three possible cases:
\begin{enumerate}[(a)]\item $S(x)=A\setminus \{u\}$, $S(y)=A\setminus \{w\}$, \item either $S(x)=A\setminus \{u\}$, $S(y)=A$ or  $S(x)=A$, $S(y)=A\setminus \{w\}$; \item  $S(x)=S(y)=A$. \end{enumerate} In other words: in case (a) we have $x_{u}=y_{w}=0$; in case (b) either $x_{u}=0$, $y_{w}>0$ or $x_{u}>0$, $y_{w}=0$; and in case (c) we have $x_{u},y_{w}>0$.

 In case (a) we have $\theta_{D}((x\wedge y)_{D})=\theta_{D}(x_{D})=\theta_{D}(y_{D})=0$ for every $D$ containing $\{u,w\}$ so (\ref{eq:sumthetas2}) becomes
$$\sum_{D:\{u,w\}\subseteq D\subseteq A}\theta_{D}((x\vee y)_{D})\geq 0.$$ By choosing different pairs $x,y$, this can be equivalently rewritten as
$$
\sum_{D:\{u,w\}\subseteq D\subseteq S(x)}\theta_{D}(x_{D})\geq 0\quad\mbox{for all }x\in\mathcal X
$$
and the sum on the left is precisely $\gamma_{uw}(x)$.

In case (b), if $S(x)=A\setminus \{u\}$, $S(y)=A$, (\ref{eq:sumthetas2}) becomes
$$\sum_{D:\{u,w\}\subseteq D\subseteq A}(\theta_{D}((x\vee y)_{D})-\theta_{D}(y_{D}))\geq 0,$$ where $A=S(x\vee y)=S(y)$, which is equivalent to $\gamma_{uw}$ being nondecreasing on $\cX^{A}$.

Finally, in case (c), all $x\vee y$, $x\wedge y$, $x$, $y$ have the same support. Thus $\gamma_{uw}$ must be supermodular over each $\cX^{A}$.
\end{proof}

As a special case we recover the characterization of binary $\mtp$ distributions in \cite{forcina2000}.

\begin{cor}\label{cor:forcina}
Let $P$ be a binary distribution with $$\log p(x) = \sum_{D: D\subseteq S(x)}\theta_{D}$$ using the convention that $\theta_{D} = \theta_{D}(\mathbf 1_{D})$. Then $P$ is $\mtp$ if and only if for all $A$ with $|A|\geq 2$ and all $\{u,w\}\subseteq V$ we have
$$
\sum_{D:\{u,w\}\subseteq D\subseteq A}\theta_{D}\;\;\geq \;\;0.
$$
\end{cor}

\begin{proof}
In the binary case each $\cX^{A}$ has only one element and so the only constraint from Theorem~\ref{thm:fuwmtp} is the non-negativity constraint.
\end{proof}
\begin{ex}\rm Let $X=(X_1=1_A,X_2=1_B,X_3=1_C)$ be the vector of binary indicator functions of events $A$, $B$, $C$. Reichenbach~\cite[p.\ 190]{reichenbach:56} (using a different notation) says that \emph{an event $B$ is causally between $A$ and $C$}  if $P(C\cd B\wedge A)=P(C\cd B)$ and  further
\[1>P(C\cd B)>P(C\cd A)>P(C)>0,\]\[1>P(A\cd B)>P(A\cd C)>P(A)>0.\] Equivalently, as defined
in \citep{chvatal:wu:12}, $B$ is causally between $A$ and $C$ if the following hold:
\begin{eqnarray}P(A\wedge C) &> &P(A)P(C)\label{eq:poscorr}\\ P(A\cd B)&>& P(A\cd C)\label{eq:alargest}\\P(C\cd B) &>& P(C\cd A)\label{eq:clargest}\\
P(A\wedge C\cd B)& =&P(A\cd B)P(C\cd B),\label{eq:indep}\\
P(\neg A\wedge B)>0,&&P(\neg C\wedge B)>0.\label{eq:positivity}\end{eqnarray}
In general, causal betweenness does not imply $\mtp$;  if  we let $p_{101}=0$, $p_{000}=4/10$, and $p_{ijk}=1/10$ for the remaining six possibilities, $B$ is causally between $A$ and $C$, but $X$ is not $\mtp$ since $0=p_{101}p_{000}<p_{100}p_{001}$.

However,   \emph{if $P(X=x)>0$ for all $x$ and $B$ is causally between $A$ and $C$, then $P$ is $\mtp$.}
To see this, we expand $P$ in log-linear interaction parameters to get
\begin{equation}\label{eq:indeplogl}1_A\cip 1_C\cd (1_B=1)\iff \theta_{AC}+\theta_{ABC}=0.\end{equation}
Further, a simple but somewhat tedious calculation using (\ref{eq:indeplogl}) yields that (\ref{eq:alargest}) and (\ref{eq:clargest}) hold if and only if
\[\theta_{AB}> \theta_{AC},\quad \theta_{BC}> \theta_{AC},\] which in combination with (\ref{eq:indeplogl}) gives that (\ref{eq:alargest})--(\ref{eq:indep}) hold if and only if
\begin{equation}\label{eq:loglinbetween}\theta_{AB}+\theta_{ABC}> 0,\quad \theta_{BC}+\theta_{ABC}> 0,\quad \theta_{AC}+\theta_{ABC}= 0.\end{equation}
Thus, Corollary~\ref{cor:forcina} ensures that $P$ is $\mtp$.

Conversely, if $P(X=x)>0$ for all $x$ and $P$ is $\mtp$, the condition $1_A\cip 1_C\cd (1_B=1)$ implies ``weak causal betweenness", i.e.\ $P$ satisfies the  inequalities (\ref{eq:poscorr})--(\ref{eq:clargest}) with $>$ replaced by $\geq$; this is true because the weak form of (\ref{eq:alargest}) and (\ref{eq:clargest}) follows from the weak form of (\ref{eq:loglinbetween}), and (\ref{eq:poscorr}) in its weak form expresses that $\cov(X_1,X_3)\geq 0$, which is also a consequence of $\mtp$.

Finally, if $P(X=x)>0$ for all $x$, $P$ is $\mtp$, and the independence graph of $P$ is $1\mbox{---}2\mbox{---}3$, then $B$ is causally between $A$ and $C$, as then $P$  is faithful by Theorem~\ref{th:Faith}, which ensures that inequalities are strict.
\qed
\end{ex}

\subsection{Conditional Gaussian distributions}\label{sec:cgdist}
In this section, we study  \emph{CG-distributions} satisfying the $\mtp$ property. The density of a CG-distribution is given by specifying a strictly positive distribution $p(i)$ over the discrete variables for $i\in \cX_\Delta$. Then the joint density $f(x)=f(i,y)$ is determined by specifying $f(y\cd i)$ to be the density of a Gaussian distribution $\mathcal N_{\Gamma}(\xi(i),\Sigma(i))$, where $\xi(i)\in \R^{\Gamma}$ is the mean vector and $\Sigma(i)$ is the covariance  matrix.  CG-distributions can also be represented by the set of canonical characteristics $(g,h,K)$ where
$$
\log f(x)= \log f(y,i)=g(i)+h(i)^{T}y-\frac{1}{2}y^{T} K(i)y;
$$
see~\cite{lau96}. Here $K(i)=\Sigma^{-1}(i)$ is the conditional concentration matrix.
We shall say that a function $u(i)$ is \emph{additive} if it has the form \[u(i)=\sum_{\delta\in\Delta}\alpha_\delta(i_\delta).\]
Before we characterize CG-distribution with the $\mtp$ property, we need a small lemma.
\begin{lem}\label{lem:modular}A function $u: \cX_\Delta\to \R$ is additive if and only if it is modular.
\end{lem}
\begin{proof} If $u$ is additive, then it is clearly modular. To show the converse, we make a log-linear expansion of $u$ as in (\ref{eq:loglinear})
\[u(i)=\sum_{D\in \mathcal{D}}\eta_D(i).\] We shall show that if $u$ is modular, then $\eta_D(i)=0$ whenever $|D|\geq 2$.
If for $C\subseteq V$ we let
\[w_C(i)=u(i_C,0_{V\setminus C}),\] it follows from the M\"obius inversion lemma, also known as inclusion--exclusion ( e.g.\ p.\ 239 of \cite{lau96}) that
\[\eta_D(i)=\sum_{A:A\subseteq D}(-1)^{|D\setminus A|} w_A(i).\]
If $|D|\geq 2$ we can for distinct $u,v\in D$ rewrite this as
\begin{eqnarray*}
\lefteqn{ \eta_D(i) = }\\& &\sum_{A: A\subseteq D\setminus\{u,v\}}(-1)^{|D\setminus A|}
\left\{w_{A\cup \{u,v\}}(i)-w_{A\cup \{u\}}(i)-w_{A\cup \{v\}}(i)+w_{A}(i)\right\}.
\end{eqnarray*}
If $u$ is modular, all terms inside the curly brackets are zero and hence $u$ is additive. \end{proof}

\begin{prop}\label{prop:cgcan}A CG-distribution $P$ with canonical characteristics $(g,h,K)$ is $\mtp$ if and only if
\begin{itemize}
\item [(i)] $g(i)$ is supermodular;
\item [(ii)] $h(i)$ is additive and non-decreasing;
\item [(iii)] $K(i)=K$ for all $i$ where $K$ is an $M$-matrix.
\end{itemize}
\end{prop}
\begin{proof}
By Proposition~\ref{KarlinRinott80}, a CG distribution is $\mtp$ if and only if it satisfies
$$
f(y\wedge z,i\wedge j)f(y\vee z,i\vee j)\geq f(y,i) f(z,j)
$$
for cases where $(y,i)$ and $(z,j)$ differ on two coordinates. Suppose first that $i=j$ and $y,z$ differ on two coordinates. Then we equivalently need to check whether
$$
f(y\wedge z\cd i)f(y\vee z\cd i)\geq f(y\cd i) f(z\cd i).
$$
Since $f(y\cd i)$ is the density of a Gaussian distribution, this inequality holds for every $y,z\in \R^{\Gamma}$ and $i$ if and only if each $K(i)$ is an $M$-matrix.
\medskip

If $i,j$ and $y,z$ both differ on one coordinate then without loss of generality we can assume $i<j$ and $y>z$ so that $i=i\wedge j$ and $z=y\wedge z$. In this case we need to show that
\begin{equation}\label{eq:aux111}
\log f(z,i)+\log f(y, j)\geq \log f(y,i)+\log f(z,j).
\end{equation}
Write $y=z+t e_{k}$ for some $t>0$, where $e_{k}$ is a unit vector in $\R^{\Gamma}$. Then equivalently
$$\frac{1}{t}(\log f(z+te_{k}, j)-\log f(z,j))\geq \frac{1}{t}(\log f(z+te_{k},i)-\log f(z,i)).
$$
Since this holds for every $t>0$, we can take the limit $t\to 0$, which implies that necessarily
$$
\nabla_{z} \log f(z,j) \geq \nabla_{z} \log f(z,i) \quad\mbox{for all }z\in \R^{\Gamma}, i<j\in \Delta.
$$
Since $\nabla_{y} \log f(y,i)= h(i)-K(i)y$, this
is equivalent to
$$
h(j)-h(i)-(K(j)-K(i))z\geq 0.
$$
The function on the left-hand side is linear in $z$ and thus this holds for every $z$ if and only if $K(j)=K(i)$ for every $i,j$ and $h(i)$ is non-decreasing in $i$.
%\medskip

If $y=z$ and $i,j$ differ on two coordinates, using all the conditions that have been already proven to be necessary we need to check that
\begin{equation}\label{eq:modularityrelation}
(g(i\wedge j)+g(i\vee j)-g(i)-g(j))+(h(i\wedge j)+h(i\vee j)-h(i)-h(j))^{T}z\geq 0.
\end{equation}
This can hold for every $z$ only if $h$ is modular, i.e.\ \begin{equation}\label{eq:modular}h(i\wedge j)+h(i\vee j)-h(i)-h(j)=0 \mbox{ for all $i,j$.}\end{equation}
Now if (\ref{eq:modular}) holds, (\ref{eq:modularityrelation}) holds if and only if $g(i)$ is super-modular. 
By Lemma~\ref{lem:modular}, $h$ is additive, which concludes the proof.\end{proof}

Proposition~\ref{prop:cgcan} gives a simple condition for CG distributions to be $\mtp$ in terms of their canonical characteristics. This also implies that the \emph{moment characteristics} $(p,\xi,\Sigma)$ have simple properties.

\begin{prop}\label{prop:cgmom}If a CG-distribution is $\mtp$, its moment characteristics $(p,\xi,\Sigma)$ satisfy
\begin{itemize}
\item [(i)] $p(i)$ is $\mtp$;
\item [(ii)] $\xi(i)$ is additive and non-decreasing;
\item [(iii)] $\Sigma(i)=\Sigma$ for all $i$ and all elements of $\Sigma$ are non-negative.
\end{itemize}
\end{prop}
\begin{proof} If the CG distribution is $\mtp$, (iii) follows directly from Proposition~\ref{prop:cgcan}. The condition (i) follows since marginals of $\mtp$ distributions are $\mtp$  and (ii) follows  from (ii) of Proposition~\ref{prop:cgcan} since $\xi(i)=\Sigma h(i)$ and $\Sigma$ has only non-negative elements.
\end{proof}
Thus, $\mtp$ CG-distributions are in particular \emph{homogeneous} --- $\Sigma(i)$ constant in $i$ --- and \emph{mean-additive} \cite{edwards,lau96}.
Note that the converse of Proposition~\ref{prop:cgmom} is not true since $\xi(i)$ can be non-increasing and $h=K\xi(i)$ decreasing, even when $K$ is  an M-matrix.

Finally we make expansions of $g(i)$ as in Section~\ref{sec:loglinear}:
\[g(i)=\sum_{D\in \mathcal{D}}\lambda_D(i),\qquad
\gamma^g_{uw}(i)=\sum_{D:\{u,w\}\subseteq D\subseteq S(i)} \lambda_{D}(i)\]
and recall that $\mathcal{I}^A=\{i: S(i)=A\}$.
We then have the following alternative formulation of Proposition~\ref{prop:cgcan}:
\begin{cor}\label{cor:cg}A CG-distribution is $\mtp$ if and only if
\begin{itemize}
\item [(i)] For all $A\subseteq V$ with $|A|\geq 2$ and all $u,w\in V$, the functions $\gamma^g_{uw}(i)$ are supermodular and non-decreasing over each $\mathcal{I}^A$;
\item [(ii)] The function $h(i)$ is additive \[h(i)=\sum_{\delta\in\Delta}\alpha_\delta(i_\delta)\] with non-decreasing components $\alpha_\delta(i_\delta)_v$;
\item [(iii)] There exists an $M$-matrix $K$ such that $K(i)=K$ for all $i$.
\end{itemize}

\end{cor}
\begin{proof}
The proof of Theorem~\ref{thm:fuwmtp} does not use that $\log p$ is the logarithm of a probability distribution; hence the corresponding conclusions also apply to the expansion of $g$.
\end{proof}
Note that if $i$ is binary, condition (i) of the Corollary  simplifies as in Corollary~\ref{cor:forcina} to the condition that for all $A\subseteq V$ and all $u\neq w\in A$ we have
$$
\sum_{D:\{u,w\}\subseteq D\subseteq A}\lambda_{D}\;\;\geq \;\;0.
$$

\section{Discussion}
\label{sec:discussion}
In this paper, we showed that $\mtp$ distributions enjoy many important properties related to conditional independence; in particular, an independence model generated by an $\mtp$ distribution is an upward-stable and singleton-transitive compositional semigraphoid which is faithful to its concentration graph if it has coordinate-wise connected support.

We illustrated with several examples that $\mtp$ models are useful for data analysis. The $\mtp$ constraint seems restrictive when no conditional independences are taken into account. However, the picture changes and the $\mtp$ constraint becomes less restrictive when imposing conditional independence constraints in the form of Markov properties.

An important property of $\mtp$ models, which is of practical relevance, is that the positive conditional dependence of two variables given all remaining variables implies a positive dependence given any subset of the remaining variables. This is a highly desirable feature, especially when results of follow-up empirical studies are to be compared with an earlier comprehensive study, and the studies coincide only in a subset of core variables.

More generally, in any $\mtp$ concentration graph model, dependence reversals cannot occur. This undesirable, worrying feature has been described and studied under different names depending on the types of the involved variables, for instance as \emph{near multicollinearity} for continuous variables, as the \emph{Yule-Simpson paradox} for discrete variables, or as the effects of \emph{highly unbalanced experimental designs} with explanatory discrete variables for continuous responses. It is a remarkable feature of $\mtp$ distributions that such dependence reversals are absent. 

These observations suggest that it would be desirable to develop further methods for hypothesis testing and estimation under the $\mtp$ constraint as done for the binary case in \cite{forcina2000}. Our results may also be applied not to the joint distribution of all variables, but only to the joint distribution of a subset of the variables, given a fixed set of level combinations of the remaining variables. This is particularly interesting in empirical studies, where a set of possible regressors and background variables is manipulated to make the studied groups of individuals as comparable as possible; for instance by selecting equal numbers of persons for fixed level combinations of some features, by proportional allocation of patients to treatments, by matching or by stratified  sampling. In such situations, not much can be inferred from the study results about the conditional distribution of the manipulated variables given the responses. However, the $\mtp$ property of the joint conditional distribution of the responses given the manipulated variables could be essential.

\section*{Acknowledgements} This article represents research initiated at the American Institute of Mathematics workshop ``Positivity, Graphical Models, and the Modeling of Complex Multivariate Dependencies" in October 2014. The authors thank the AIM and NSF for their support and are indebted to other participants of the workshop for useful and constructive discussions. We are especially grateful to H\'el\`ene Massam for her constructive remarks and encouragement. Thanks are also due to Heiko Becher for permission to use his data in
Example~\ref{ex:laryngeal}, and to Antonio Forcina and two anonymous referees for their helpful comments.

\clearpage
\bibliographystyle{imsart-number}
\bibliography{bib_mtp2}

\end{document}